\documentclass[12
pt,twoside]{amsart}
\usepackage{amssymb}
\usepackage{a4wide}
\usepackage[latin1]{inputenc}
\usepackage[T1]{fontenc}
\usepackage{times}

\def\hpq0{h^{p,q}_{\leq 0}}
\def\Hpq0{\H_{\leq 0}^{p,q}}

\def\dbar{\bar\partial}
\def\ddbar{\partial\dbar}

\def\R{{\mathbb R}}

\def\C{{\mathbb C}}

\def\H{{\mathcal H}}

\def\Re{{\rm Re\,  }}

\def\be{\begin{equation}}
\def\ee{\end{equation}}

\newtheorem{thm}{Theorem}[section]

\newtheorem{cor}[thm]{Corollary}
\newtheorem{prop}[thm]{Proposition}

\theoremstyle{definition}

\newtheorem{df}{Definition}

\theoremstyle{remark}

\newtheorem{preremark}{Remark}
\newtheorem{preex}{Example}

\numberwithin{equation}{section}

\begin{document}

\begin{abstract} We give a general inequality for Bergman kernels of Bergman spaces defined by convex weights in $\C^n$. We also discuss how this can be used in Nazarov's proof of the Bourgain-Milman theorem, as a substitute for H\"ormander's estimates for the $\dbar$-equation.
\end{abstract}

\title[]
{ Bergman kernels for Paley-Wiener  spaces and Nazarov's proof of the Bourgain-Milman theorem.}

\author[]{ Bo Berndtsson}

\bigskip

\maketitle

\section{Introduction}
If $\phi$ is a plurisubharmonic function in a domain $D$ in $\C^n$, the Bergman space $A^2(D, e^{-\phi})$ is the Hilbert space of holomorphic functions in $D$ that are square integrable against the weight $e^{-\phi}$;
$$
A^2=A^2(D,e^{-\phi})=\{f\in H(D); \int_D |f|^2 e^{-\phi} d\lambda<\infty\}.
$$
The (diagonal) Bergman kernel for $A^2$ is the function
$$
B(z):= \sup_{f\in A^2} |f(z)|^2/\|f\|^2.
$$
In this note we are mainly interest in the case when $\phi$ is convex and $D=\C^n$. However, we do allow our convex functions to attain the value $+\infty$, and in that case we are in effect only integrating over the set where $\phi$ is finite, and we only require $f$ to be holomorphic in the interior of that set. With this understanding we omit the reference to the domain in the Bergman space, and write $A^2=A^2(e^{-\phi})$. Our main result is as follows.
\begin{thm} Let $\phi_1$ and $\phi_2$ be convex functions on $\R^n$ and let $B$ be the Bergman kernel for the space $A^2(e^{-\phi_1(x)-\phi_2(x)})$. Let also $B'$ be the Bergman kernel for the space $A^2(e^{-\phi_1(x)-\phi_2(y)})$.
Then, if $\phi_1$ and $\phi_2$ are symmetric ($\phi_j(-x)=\phi_j(x)$), 
$$
B'(0)\leq C^n B(0),
$$
where
$$
 \log C= \frac{1}{\pi}\int_{-1}^1\frac{-\log s^2}{1+s^2} ds.
  $$

   If $\phi_1=\phi_2=\phi$, we may take $C=2$ if $\phi$ is homogenous of order 2, and $C=1.604$ if $\phi$ is homogenous of order 1.
  \end{thm}

The point of the theorem is that the spaces   $A_0:=A^2(e^{-\phi_1(x)-\phi_2(x)})$  and 
$A_1:=A^2(e^{-\phi_1(x)-\phi_2(y)})$, are  rather different and it is not a priori clear that their Bergman kernels should be in any way related. While $A_0$ is defined by a weight function that only depends on $x$, the weight in the definition of $A_1$ decreases in all directions. Thus, e. g., the function $f=1$ lies in $A_1$ but not in $A_0$. Using that function in the definition of the Bergman kernel we get immediately
$$
B'(0)\geq (\int e^{-\phi_1} \int e^{-\phi_2})^{-1},
$$
and the theorem implies that the same estimate holds for $B(0)$, up to a constant:
\begin{cor} With the notation of the theorem
\be
B(0)\geq C^{-n} (\int e^{-\phi_1} \int e^{-\phi_2})^{-1}
\ee
\end{cor}
Our principal motivation (or inspiration) for the theorem comes from Nazarov's proof (\cite{Nazarov}) of the Bourgain-Milman theorem, \cite{Bourgain-Milman}. Nazarov's idea was to use Bergman kernels for Paley-Wiener spaces (see section 3) to prove inequalities for volumes of convex bodies. His main result is basically what we have formulated as a corollary, in the special case when $\phi_1=\phi_2= I_K$, the indicator of a symmetric convex body (defined as zero on $K$ and infinity outside). The inequality 
(1.1) then becomes
\be
B(0)\geq C^{-n} |K|^{-2},
\ee
where $|K|$ means the volume of $K$. He then coupled this estimate with an estimate from above
\be
B(0)\leq n! C_1^n |K^\circ||K|^{-1},
\ee
where $K^\circ$ is the polar body of $K$. The result is that
$$
|K||K^\circ|\geq C_2^n/n!,
$$
with $C_2$ an absolute constant. This is the Bourgain-Milman inequality. The famous Mahler conjecture says that $C_2$ can be taken equal to $4$. This is so far open. The best result on the value of the constant so far, $C_2=\pi$, is due to Kuperberg, \cite{Kuperberg}.

Nazarov used H\"ormander's $L^2$-estimates for the $\dbar$-equation to prove (1.2). Here we will use a different argument based on plurisubharmonic variation of Bergman kernels from \cite{Berndtsson2}. The inequality (1.3) is also non trivial, but much easier than (1.2), and completely elementary. We will give the corresponding inequality for convex functions in section 4 (this is  based on the master's thesis of Hultgren, \cite{Hultgren}).

In the next section we will give the proof of Theorem 1.1. In section 3 we give some generalities on Bergman and Paley-Wiener spaces, and in section 4 we discuss estimates from above of Bergman kernels ( in both these sections, we  again largely  follow \cite{Hultgren}). In the last section we will discuss the values  of the constant one can obtain by this method -- they are somewhat worse than Kuperberg's estimate, but just barely better than Nazarov's. 

It should be said that the only new result in this note is Theorem 1.1. The rest of the material is at least essentially known, but we include it in an attempt to give an easy introduction to this very interesting circle of ideas. In a companion paper, \cite{B}, we also give a variation of Kuperberg's proof, using different ideas from complex analysis. We also refer to \cite{Blocki} for another analysis of Nazarov's proof, from a different angle. 

Finally I would like to thank Bo'az Klartag and Yanir Rubinstein for very stimulating discussions on these matters.

\section{An estimate for Bergman kernels}

We are looking at Bergman spaces on $\C^n$ defined by weight functions $e^{-\phi(x)}$ where $\phi$ is convex on $\R^n$. We may assume that $\phi(0)=0$ and  also assume that $\phi(x)=\phi(-x)$.
The main technical difficulty in Nazarov's paper is to prove that for $\phi=I_K$, the Bergman kernel at the origin is bounded from below by a constant depending on $n$  times $|K|^{-2}$. This is where H\"ormander's theorem comes in.

Here we follow  another approach to prove such things in a somewhat more general setting. Let $\phi_1(x)$ and $\phi_2(x)$ be two convex functions on $\R^n$ satisfying the hypotheses above.  We consider the Bergman space $A^2(e^{-\phi_1(x)-\phi_2(x)})$ of holomorphic functions satisfying
$$
\int |f(x+iy)|^2 e^{-\phi_1(x)-\phi_2(x)} dxdy<\infty.
$$

Let $\zeta\in \C$ and define
$$
\|f\|^2_\zeta=\int_{\C^n} |f|^2 e^{-\phi_1(\zeta z)-\phi_2(z)} d\lambda(z).
$$
Here we are using the notation $\phi_j(z)=\phi_j(\Re z)$.
Denote by $B_\zeta$ the Bergman kernel for this norm. 
When $\zeta=i$, the norm is
$$
\|f\|^2_i=\int_{\C^n} |f|^2 e^{-\phi_1(-y)-\phi_2(x)} d\lambda.
$$
The main point is that the  Bergman kernel for this norm is fairly easy to estimate from below since $f=1$ is a contender in the definition of the Bergman kernel. The theorem says that bounds from below of $B_i$ give bounds from below of $B_1$ which is the Bergman kernel for the norm defined by $\phi_1+\phi_2$.  

\begin{thm} For general convex functions $\phi_1, \phi_2$, symmetric and vanishing at the origin, we have 
  $$
  B_i(0)\leq C^n B_1(0),
  $$
  with
  $$
  \log C= \frac{1}{\pi}\int_{-1}^1\frac{-\log s^2}{1+s^2} ds.
  $$

  If $\phi_1=\phi_2=\phi$, we may take $C=2$ if $\phi$ is homogenous of order 2, and $C=1.604$ if $\phi$ is homogenous of order 1.
  \end{thm}

The proof depends on two things: First, $\log B_\zeta(0)=: b(\zeta)$ is subharmonic in $\C$. This is a direct consequence of the main theorem in \cite{Berndtsson2} on plurisubharmonic variation of Bergman kernels since $\phi_1(\zeta z)+\phi_2(z)$ is plurisubharmonic in $(\zeta,z)$. Second, it satisfies an estimate
$$
b(\zeta)\leq C+n\log|\zeta|^2.
$$
The second fact comes from changing variables 
$$
\|f\|^2_\zeta=|\zeta|^{-2n}\int_{\C^n} |f(z/\zeta)|^2 e^{-\phi_1(z)-\phi_2(z/\zeta)} d\lambda(z).
$$
This gives that
$$
B_\zeta(0)=|\zeta|^{2n}\check{B}_{1/\zeta}(0),
$$
where $\check{B}$ means that we have changed the roles of $\phi_1$ and $\phi_2$. (It does not matter that $f(z)$ changes to $f(z/\zeta)$ since the origin where we study the Bergman kernel remains fixed.)  
Then we note  that $\check{B}_{1/\zeta}(0)$ is bounded as $\zeta\to\infty$; it tends to the Bergman kernel for the weight $e^{-\phi_1}$.  

This means that we can apply the Poisson representation in the upper half plane to the function $b(\zeta)-n\log|\zeta|^2$, which is bounded near infinity. It is not bounded near the origin, where it has a logarithmic pole, but it is at any rate integrable on the real axis. Therefore we have  that
$$
b(i)=\log B_i(0)\leq \frac{1}{\pi}\int_{-\infty}^\infty\frac{b(s)-n\log s^2}{1+s^2} ds.
$$
What remains is to estimate $b(s)$ for $s\in \R$, and by symmetry we need only worry about $s>0$. 

If $s\leq 1$, we have
\be
\|f\|^2_s=\int |f|^2 e^{-\phi_1(sx)-\phi_2(x)} d\lambda(z)\geq \|f\|^2_1,
\ee
  since $\phi(sx)$ is increasing in $s$. Hence $b(s)\leq b(1)$.

  If $s\geq 1$ we have
  \be
\|f\|^2_s=s^{-2n}\int |f|^2 e^{-\phi_1(x)-\phi_2(x/s)} d\lambda(z)\geq s^{-2n} \|f\|^2_1,
  \ee
  for the same reason. Hence $b(s)\leq b(1)+n\log s^2$ when $s\geq 1$.
Notice that these estimates of $b(s)$ are actually sharp if $\phi_1=\phi_2$ is the  indicator function of a convex body.

Putting these estimates together we get
  $$
  b(i)\leq \frac{1}{\pi}\int_{-1}^1\frac{b(1)-n\log s^2}{1+s^2} ds+
 \frac{1}{\pi}\int_{|s|>1}\frac{b(1)}{1+s^2} ds.
$$
 This completes the proof of the first part of the theorem.

 For the second part, when $\phi_1=\phi_2=\phi$ we have, when $\phi$ is homogenous of order 2, that
 $$
 \|f\|^2_s=\int |f|^2 e^{-\phi(sx)-\phi(x)} d\lambda(z)=\int |f|^2e^{-2\phi(\sqrt{(s^2+1)/2}     x)}d\lambda(z)=
 $$
 $$
= (\frac{2}{1+s^2})^n\int |f(z/ \sqrt{(s^2+1)/2} )|^2 e^{-2\phi(x)} d\lambda(z). 
$$
This implies that
$$
B_s(0)= (\frac{s^2+1}{2})^n B_1(0),
$$
where $B_1(0)$ is now the Bergman kernel for the weight function $e^{-2\phi}$. Hence, using the Poisson integral representation again, 
 $$
  b(i)=\log B_i(0)\leq n\frac{1}{\pi}\int_{-\infty}^\infty \frac{\log(1+s^2)-\log 2}{1+s^2}ds + b(1).
  $$
  (The term $\log s^2$ gives zero  contribution to the Poisson integral.) 

  We now use that 
  $$
  \frac{1}{\pi}\int_{-\infty}^\infty \frac{\log(1+s^2)}{1+s^2}ds =2\log 2
  $$
  (it is a standard residue integral). Hence
  $$
  b(i)\leq n\log 2+b(1),
  $$
  which is what we claim. 

Finally, if $\phi_1=\phi_2=\phi$ where $\phi$ is homogenous of order 1, we write $\phi(sx)+\phi(x)=2\phi(x(|s|+1)/2 )$. The same argument as in the 2-homogenous case gives
$$
\log C= \frac{4}{\pi}\int_0^\infty \frac{\log(1+s)}{1+s^2}ds -2\log 2.
$$
By numerical calculation (e.g. Wolfram alpha) this gives $C$ slightly less than $1.604$.
 \qed 

 It is interesting to compare the second part of the theorem to the special case when $\phi(x)= |x|^2/2$. Then
  $B_i(0)=B_1(0)$. To see this, note that
 $$
 x^2=(x^2+y^2)/2 +\Re z^2/2
 $$
 Thus, if we multiply $f$ by $e^{z^2/4}$, we change the second norm to the first. Hence the Bergman kernels are equal since  $e^{z^2/4}=1$ at the origin. I don't know if the factor $2^n$ in the theorem can be omitted in general.

 Note also that    when $\phi$ is $x^2/2$                       , then
 \be
 b(\zeta)=n(\log(1+|\zeta|^2)-\log2) +b(1),
 \ee
 not only for $\zeta$ real, but for all $\zeta$. This follows by a similar argument: Writing $\zeta=\xi+i\eta$, 
 $$
 |\Re(\zeta z)|^2= \xi^2 |x|^2+\eta^2 |y|^2-2\xi\eta x\cdot y,
 $$
 which is equal to $|\zeta|^2|x|^2$ plus the real part of a holomorphic function ( of $z$) vanishing at the origin. Hence, the same argument as above gives 1.3. (Moreover, the same formula holds for any quadratic form $\phi= x^tA^tAx$, where $A$ is a real $n\times n$ matrix; this follows by changing variables $w=Az$.)

 So, when $\phi(x)=|x|^2/2$ (or a general quadratic form), the estimate in the theorem is off by a factor $2^n$. We loose this factor in the proof when we use the Poisson representation formula. The Poisson representation is an identity for harmonic functions, but for subharmonic functions we loose the Green potential part. For $\phi$ a quadratic form, the discussion above shows that the term we loose is the Green potential of the Fubini-Study form $i\ddbar \log(1+|\zeta|^2)$.

\section{Bergman spaces, Paley-Wiener spaces and logarithmic Laplace transforms.}
If $K$ is a convex body in $\R^n$, the Paley-Wiener space associated to $K$, $PW(K)$ is the space of all entire functions in $\C^n$  of the form
$$
f(z)=\int_K e^{t\cdot z} \tilde f(t) dt,
$$
where $\tilde f$ lies in $L^2(K)$. We define the norm of $f$ to be $L^2$-norm of $\tilde f$. Then $PW(K)$ becomes a Hilbert space of entire functions. Any $f$ in the space satisfies an estimate
$$
|f(x+iy)|\leq e^{h_K(x)}\|f\|\sqrt{|K|},
$$
where 
$$
h_k(x)=\sup_{t\in K} t\cdot x,
$$
is the support function of $K$ 
and $|K|$ denotes the volume of $K$. Moreover, the restriction of $f$ to the $i\R^n$ lies in $L^2$. The Paley-Wiener theorem asserts that conversely, any such function $f$ lies in the Paley-Wiener space.

We will consider a generalization of Paley-Wiener space, associated to convex functions on $\R^n$ instead of convex bodies. If $\psi$ is a convex function on $\R^n$, we denote by $PW(e^{\psi})$ the space of all holomorphic functions of the form
\be
f(z)=\int_{\R^n} e^{t\cdot z}  \tilde f(t) dt,
\ee
where 
$$
\|f\|_{PW}^2:=\int_{\R^n} |\tilde f|^2 e^{\psi}  dt <\infty.
$$
We allow $\psi$ to attain the value $+\infty$, and the $L^2$-condition should then be interpreted so that $\tilde f$ vanishes where $\psi=\infty$. Thus the classical Paley-Wiener spaces correspond to $\psi=I_K$; the convex indicator of $K$ that is zero on $K$ and infinity on the complement of $K$. 

Another particular case of generalized Paley-Wiener spaces are  Bergman spaces, where the weight depends only on the real part of $z$. Let $\phi=\phi(x)$ be a convex function on $\R^n$, and consider the Bergman space, $A^2(e^{-\phi(x)})$,  associated to $\phi$, i.e. the space of holomorphic functions such that
\be
\|f\|^2_{A^2}:=\int|f(x+iy|^2 e^{-\phi(x)} dxdy<\infty.
\ee
Again, we allow $\phi$ to attain the value $+\infty$. We then only integrate over the tube domain
where $\phi$ is finite, and $f$ is only required to be holomorphic there. Proposition 3.1 below  says that such spaces are Paley-Wiener spaces, associated to a weight function $\psi$ that is the 'logarithmic Laplace transform' of $\phi$ (see \cite{1Bo'az} and \cite{2Bo'az} for early uses of this transform). 

\begin{df} Let $\phi$ be a convex function on $\R^n$. Then the logarithmic Laplace transform of $\phi$ is the  function $\tilde\phi$  defined by
$$
e^{\tilde\phi(t)}= \int e^{2xt-\phi(x)} dx.
$$
\end{df}

 Logarithmic Laplace transforms are clearly convex, but it is also clear that a general convex function can not be written like this. For one thing, they are always real analytic in the interior of the set where they are finite, but it is also easy to see that e.g. the indicator of a convex body is not a logarithmic Laplace transform. One can verify that the sum of two logarithmic Laplace transforms is again a logarithmic Laplace transform( of the convolution of $e^{-\phi_1}$ and $e^{-\phi_2}$; Prekopa's theorem implies that the convolution is log-concave).   Thus the set of logarithmic Laplace transforms is an additive semigroup and $2\tilde\phi$ is again in the set, but $(1/2)\tilde\phi$ is in general not. 
\begin{prop} The spaces $PW(e^{\tilde\phi})$ and $A^2(e^{-\phi})$ are equal and
$$
\|f\|^2_{A^2}=(2\pi)^n\|f\|^2_{PW}.
$$
\end{prop}
Thus we see that weighted Bergman spaces where the weight only depends on $x=\Re z$ are Paley-Wiener spaces of a special kind, namely defined by weight functions $\psi$ that are logarithmic Laplace transforms.

For the proof, we first assume that $f$  lies in $PW(e^{\tilde\phi})$.  Hence, $f$ is the Fourier-Laplace transform (3.1) of a function $\tilde f$, and
$$
\int |\tilde f|^2 e^{\tilde\phi} dt<\infty.
$$
It follows that $e^{t\cdot x}\tilde f$ lies in $L^2(\R^n)$ for any $x$ in the interior of the set where $\phi(x)<\infty$. 

Let $f_x(y)=f(x+iy)$. Then Parseval's formula gives
$$
\int |f_x|^2 dy= (2\pi)^n\int e^{2t\cdot x}|\tilde f|^2 dt.
$$
Multiplying with $e^{-\phi(x)}$ and integrating with respect to $x$ we get
$$
\int |f(x+iy)|^2 e^{-\phi(x)} dx dy=(2\pi)^n\int |\tilde f|^2 e^{\tilde\phi} dt.
$$
Thus $f$ lies in $A^2(e^{-\phi(x)})$ and the norm in $A^2$ coincides with the norm in $PW(e^{\tilde\phi})$ multiplied by $(2\pi)^n$. 

Hence, the Paley-Wiener space with weight $\tilde \phi$ is at least isometrically embedded in $A^2(e^{-\phi(x)})$, and it remains only to show that it is dense. 

For this we take $f$ in $A^2$ and note  that $f_x$ lies in $L^2$ for almost all $x$ in the interior of the set where $\phi(x)<\infty$. Take one such $x_0$. By Fourier inversion we can write
$$
f_{x_0}(y)=\int e^{t\cdot (x_0+iy)} \tilde f(t) dt,
$$
with $e^{t\cdot x_0}\tilde f$ in $L^2$. Now assume that $\tilde f$ has compact support. Then
$$
f(z)=\int e^{t\cdot z}\tilde f(t) dt,
$$
since the right hand side is an entire function and it agrees with $f$ where $\Re z=x_0$. By Parseval's formula again,
$$
\int |f(x+iy)|^2dy= (2\pi)^n\int e^{2t\cdot x } |\tilde f(t)|^2 dt,
$$
and we can multiply by $e^{-\phi(x)}$ and integrate with respect to $x$ to conclude that $f$ lies in $PW(e^{\tilde\phi})$. 

Thus, to prove that the Paley-Wiener space is dense in $A^2$ it suffices to prove that the space of $f$ in $A^2$ such that $f_x$ has compactly supported Fourier transform for some $x$, is dense. For this, let $\alpha$ be a smooth, non-negative function with compactly supported Fourier transform and total integral equal to 1. Put
$$
h(z)=\int f(z+is) \alpha(s)ds.
$$
Then $h_x$ has compactly supported Fourier transform, and since
$$
|h(z)|^2\leq \int |f(z+is)|^2\alpha(s)ds,
$$
$h$ lies still in $A^2$. Hence any such $h$ lies in the Paley-Wiener space. Finally, by scaling $\alpha$ appropriately we get an approximate identity, and therefore a sequence of functions in the Paley-Wiener space that converges to $f$. This completes the proof. 

We next give an explicit formula for the Bergman kernel of $PW(e^{\psi})$.
\begin{prop} The Bergman kernel $B(z)$ for $PW(e^{\psi})$ is
$$
B(z)=\int e^{2t\cdot x -\psi(t)} dt.
$$
In particular , the Bergman kernel for $A^2(e^{-\phi})$ is 
$$
(2\pi)^{-n}\int e^{2t\cdot x-\tilde\phi(t)} dt.
$$
 \end{prop}
\begin{proof}
By the definition of the Bergman kernel we have
$$
B(x+iy)= \sup_{\tilde f}|\int e^{t\cdot (x+iy)} \tilde f(t) dt|^2,
$$
where the supremum is taken over all $\tilde f$ with
$$
\int |\tilde f|^2 e^{\psi} dt\leq 1.
$$
The first part of the proposition follows immediately from this and Cauchy's inequality. The second part follows from the first part and formula (3.4).
\end{proof}
\begin{preremark}
There is an intriguing symmetry between Propositions 3.1 and 3.2 in that the logarithmic Laplace transform appears in both places, for rather different reasons. It would be nice to have a better explanation of this than just brute computation.
\end{preremark}
Taking $\psi = I_K$, $PW(e^{-\psi})= PW(K)$ and the proposition shows that the Bergman kernel at the origin for the classical Paley-Wiener space $PW(K)$  is $|K|$. This was the starting point for Nazarov. However, one cannot work directly in $PW(K)$, since it is difficult to construct functions in that space, by methods like H\"ormander's $L^2$-estimates.  Therefore Nazarov works instead with $\psi=\tilde\phi$, where $\phi$ is an indicator function. This is a Bergman space, and H\"ormander-type methods can be applied. He uses a variant of Proposition 3.2 for the special case when $\phi=I_K$ is the indicator of a convex body. We refer to \cite{Nazarov} for a discussion of the origins of this formula.
  \section{Estimates from above of the Bergman kernel}

 Following Nazarov, we next give an estimate from above of the Bergman kernel, with the difference that we work with convex functions instead of convex bodies (see also \cite{Hultgren}). In the next theorem we take the weight to be $2\phi(x)$ instead of $\phi(x)$. This fits better with Theorem 1.1, and is also more convenient in the estimate from above.
\begin{thm} Let $B$ be the (diagonal) Bergman kernel for the space $A^2(e^{-2\phi(x)})$. Assume $\phi$ is convex and symmetric around the origin. 
Then
\be
B(0)\leq \pi^{-n}\frac{\int e^{-\phi^*}}{\int e^{-\phi}},
\ee
where $\phi^*$ is the Legendre transform of $\phi$.
\end{thm}
\begin{proof}
Let $\Phi=2\phi$. By Proposition 3.1
\be
B(0)=(2\pi)^{-n}\int e^{-\tilde\Phi},
\ee
where 
$$
e^{\tilde\Phi(t)} =\int e^{2t\cdot x-\Phi(x)} dx.
$$
Changing variables we get
$$
e^{\tilde\Phi(t)}= 2^{-n}\int e^{t\cdot x-\Phi(x/2)} dx.
$$
Let $x=y+\xi$ where $y$ is constant and $\xi$ is a new variable in the integral. By convexity, 
$$\Phi((y+\xi)/2)\leq (\Phi(y)+\Phi(\xi))/2=\phi(x)+\phi(\xi),
$$
so
$$
e^{\tilde\Phi(t)}\geq  2^{-n} e^{t\cdot y-\phi(y)}\int e^{t\cdot \xi-\phi(\xi)} d\xi .
$$
Since $\phi$ is symmetric, the gradient with respect to $t$ of the integral in the right hans side vanishes for $t=0$, so this is a minimum point for the function defined by the integral:
$$
\int e^{t\cdot \xi-\phi(\xi)} d\xi\geq \int e^{-\phi(\xi)} d\xi.
$$
Taking the supremum over all $y$ we get
$$
e^{\tilde\Phi(t)}\geq  2^{-n}e^{\phi^*(t)}\int e^{-\phi}dx.
$$
Inserting this in formula (4.2) we get (4.1).
\end{proof}

We can now combine this with Corollary 1.2 (with $\phi_1=\phi_2=\phi)$:
\begin{thm} If $\phi$ is a symmetric convex function
$$
C^{-n} \pi^n\leq \int e^{-\phi}\int e^{-\phi^*}
$$
where $C$ is given in Theorem 1.1. If $\phi$ is homogenous of order 2, we can take $C=2$.
\end{thm}
This estimate is not optimal. The theorem of Kuperberg, \cite{Kuperberg} suggests that one can remove the factor $C^{-n}$ , see \cite{B} for a proof of this. The function version of Mahler's conjecture says that one can take the left hand side equal to $4^{n}$. 
    \section{Convex bodies and numerical values.}

In this section 
we translate Theorem 4.2 to an estimate for volumes of convex bodies and compare the result obtained to earlier results. 

Take $\phi^*=I_K$, the indicator function of a symmetric convex body $K$. Then $\phi= h_K$,
the support function of $K$. Hence
$$
\int e^{-\phi(t)}=\int_{\R^n}dt\int_{h_K(t)}^\infty e^{-s}ds=\int_0^\infty e^{-s}\int_{h_K(t)<s}dt=
|K^\circ| \int_0^\infty s^ne^{-s} ds=n!|K^\circ|,
$$
since $sK^\circ=\{ h_K<s\}$.

Hence Theorem 4.2 gives the inequality
$$
|K||K^\circ|\geq  C^{-n}\pi^n/n!,
$$
which is worse than Kuperberg's estimate by a factor $C^{-n}$, $C=1.604$. Nazarov's estimate corresponds to $C=(4/\pi)^2$, which is roughly 1.62. Thus our estimate is better than Nazarov's, but not by much.

\end{document}